\documentclass[11pt,leqno]{amsart}

\usepackage{graphicx}
\usepackage[margin=1.3in]{geometry}
\usepackage{amsmath}
\usepackage{amssymb}
\usepackage{amsthm}
\usepackage{amsfonts}
\usepackage{mathrsfs}
\usepackage[mathscr]{euscript}
\usepackage[latin1]{inputenc}
\usepackage[numbers]{natbib}
\usepackage[usenames, dvipsnames]{color}
\usepackage{enumerate}
\usepackage{enumitem}
\usepackage[scriptsize,hang,raggedright]{subfigure}
\usepackage[cmyk]{xcolor}
\usepackage{mathtools}
\usepackage{esint}

\usepackage{hyperref}
\usepackage{pdfsync}
\usepackage{dsfont}

\usepackage{bbm}
\usepackage{multirow}

\allowdisplaybreaks

\numberwithin{equation}{section}

\newtheorem{theorem}{Theorem}[section]
\newtheorem{proposition}[theorem]{Proposition}
\newtheorem{lemma}[theorem]{Lemma}
\newtheorem{corollary}[theorem]{Corollary}

\theoremstyle{remark}

\theoremstyle{definition}

\newtheorem{example}[theorem]{Example}

\newcommand{\e}{\varepsilon}
\newcommand{\R}{\mathbb{R}}

\newcommand{\ba}{\begin{array}}
\newcommand{\ea}{\end{array}}

\newcommand{\tld}[1]{\widetilde{#1}}
\newcommand{\bthm}{\begin{theorem}}
\newcommand{\ethm}{\end{theorem}}
\newcommand{\bprop}{\begin{proposition}}
\newcommand{\eprop}{\end{proposition}}
\newcommand{\blemma}{\begin{lemma}}
\newcommand{\elemma}{\end{lemma}}
\newcommand{\bexmpl}{\begin{example}}
\newcommand{\eexmpl}{\end{example}}

\newcommand{\beqn}{\begin{equation}}
\newcommand{\eeqn}{\end{equation}}
\newcommand{\beqns}{\begin{equation*}}
\newcommand{\eeqns}{\end{equation*}}

\newcommand{\ip}[2]{\left\langle #1 , #2 \right\rangle}

\newcommand{\pr}{\prime}

\newcommand{\Stwo}{{\mathbb{S}^2}}

\newcommand{\Htwo}{\mathcal{H}^2}

\newcommand{\A}{\mathcal{A}}
\newcommand{\Se}{\mathcal{S}}
\newcommand{\M}{\mathcal{M}}
\newcommand{\Y}{\mathcal{Y}}

\newcommand{\D}{\mathcal{D}}
\newcommand{\Inv}{\mathcal{I}}

\renewcommand{\leq}{\leqslant}
\renewcommand{\geq}{\geqslant}

\definecolor{mygreen}{rgb}{0.1,0.75,0.2}

\newcommand{\eps}{\epsilon}

\newcommand{\E}{\mathsf{E}}

\newcommand{\F}{\mathsf{F}}
\newcommand{\G}{\mathsf{G}}
\newcommand{\He}{\mathsf{H}}

\newcommand{\OK}{\mathsf{OK}}

\newcommand{\calK}{\mathcal{K}}

\DeclareMathOperator{\per}{per}

\DeclareMathOperator{\loc}{loc}

\mathtoolsset{showonlyrefs}

\title[Singular Perturbation of an Elastic Energy with a Singular Weight]{Singular Perturbation of an Elastic Energy with a Singular Weight}

\author{Oleksandr Misiats}
\address{Department of Mathematics and Applied Mathematics, Virginia Commonwealth University, Richmond, VA}
\email{omisiats@vcu.edu}
\author{Ihsan Topaloglu}
\address{Department of Mathematics and Applied Mathematics, Virginia Commonwealth University, Richmond, VA}
\email{iatopaloglu@vcu.edu}
\author{Daniel Vasiliu}
\address{Department of Mathematics and Applied Mathematics, Virginia Commonwealth University, Richmond, VA}
\email{dvasiliu@vcu.edu}
\date{\today}

\begin{document}

\begin{abstract} We study the singular perturbation of an elastic energy with a singular weight. The minimization of this energy results in a multi-scale pattern formation. We derive an energy scaling law in terms of the perturbation parameter and prove that, although one cannot expect periodicity of minimizers, the energy of a minimizer is uniformly distributed across the sample. Finally, following the approach developed by Alberti and M\"{u}ller \cite{AlMu2001} we prove that a sequence of minimizers of the perturbed energies converges to a Young measure supported on functions of slope $\pm 1$ and of period depending on the location in the domain and the weights in the energy.
\end{abstract}

\maketitle

\baselineskip=13pt

\section{Introduction}\label{sec:intro}

In this paper we study minimizers of the singularly perturbed energy functionals of the form
	\beqn\label{eqn:energy}
	 \E_\e(u) = \int_{0}^1\left( \e \, t^\alpha\,  |u^{\pr\pr}| + t^{-\beta}\,u^2\right)\,dt
	\eeqn
over the admissible class
	\[
		\A = \big\{u\in H^2([0,1])  \colon |u^\pr|=1,\, u(0)=u(1)=0\big\},
	\]
where $\alpha$ and $\beta$ are constants. 

Functionals of this type appear in models of solid-to-solid phase transitions. They can be used to describe the multi-scale pattern formation of distinct phases and to understand the length scale of fine structures as well as their periodicity. The functional \eqref{eqn:energy} without weights (i.e., $\alpha=\beta=0$) was analyzed by M\"{u}ller in \cite{Mul93} where he argued that when $\e>0$ is sufficiently small, the energy order of magnitude scales as $\e^{2/3}$ and the minimizers are periodic with period proportional to $\e^{1/3}$. Given the assumption of competing weights in front of the elastic and surface energy terms, the periodicity of minimizers of \eqref{eqn:energy} can no longer be anticipated. 

Specifically, the effect of the weights in the energy \eqref{eqn:energy} is that the minimizing fine structures have a priori unknown multi-scale behavior which depends on the location in the domain $[0,1]$. In \cite{AlMu2001}, Alberti and M\"{u}ller introduced a novel approach by extending  the classical $\Gamma$-convergence methods to rigorously analyze variational problems with two distinct length scales. As in formal asymptotics they introduced a slow and a fast scale and investigated the rescalings $R_s^\e u^\e(t) = \e^{-1/3}\,u^\e(s+\e^{1/3}t)$ of minimizers $u^\e$. This approach facilitated the derivation of a variational problem reformulated in terms of the Young measure that was generated by the maps $s\mapsto R_s^\e u^\e$. It was argued that this particular Young measure $\nu_s$ represents the probability that $R_s^\e u^\e$ assumes a certain value in a small neighborhood of $s$ in the limit $\e\to 0$. In other words, the measure $\nu_s$ gives the probability to find a certain pattern on the scale $\e^{1/3}$ near the point $s$ and is supported on micropatterns. 

In the seminal paper \cite{AlMu2001} two-scale energies were considered under the assumption that the weight of the elastic term is in $L^\infty$. Here we extend on their results by allowing unbounded weights for the bulk energy term, as well as an additional weight in the surface energy term. The first step in our proceedings is to construct an explicit upper bound on the energy of minimizers $\E_\e(u^\e)$ and to show that the energy is of order $\e^{2/3}$ when $\e\to 0$ , i.e.,
\[
C_1\,\e^{2/3} \leq \E_\e(u^\e) \leq C_2\, \e^{2/3}
\]
for some constants $C_1,\, C_2>0$ provided $\beta<3$. We prove this scaling law in Section \ref{sec:scaling}. This construction requires a fine analysis since near $t=0$ the singularity of the weight $t^{-\beta}$ needs to be controlled. 

As noted above, due to the weights in the energy functionals, one should not expect periodicity of minimizers; however, we are still able to obtain that the energy of a minimizer is distributed uniformly throughout the sample. This result is reminiscent of the branching phenomena that occurs near the austenite interface in some models of martensitic phase transitions (cf. \cite{Co2000,KoMu94}). As detailed in Section \ref{sec:distribution}, if we denote by $\varphi(x)$ the energy contribution of a minimizer on the interval $[0,x]$, then we obtain that
	\[
		\varphi(x) \leq C x \e^{2/3}
	\]
provided $x \geq C \e^{2/(9 - 3\beta)}$ and $\beta<3$ with $2\alpha\geq \beta$. This implies that, despite our upper bound construction for the energy scaling law requires more oscillations on the edge where the weight of the bulk term has a singularity, the energy distribution is still uniform across the domain.

Finally, in Section \ref{sec:limit}, we adapt the approach developed in \cite{AlMu2001} and, quite similarly, we identify the asymptotic limit of minimizers of energies $\E_\e$ and their diffuse-level counterparts. In particular, for $\alpha,\, \beta>0$ with $\beta-2\alpha<3$, we prove that for a sequence of minimizers $\{u^\e\}_{\e>0}$ the Young measure which arises as the limit of the maps $s\mapsto R_s^\e u^\e$ as $\e\to 0$ is supported on the set of all translations of sawtooth functions with slope $\pm 1$ and period a constant multiple of $s^{(\alpha+2\beta)/6}$. The arguments and proofs in this part of our paper follow mostly from those in \cite{AlMu2001}; however, due to the inclusion of a singular weight in the energy some modifications are required. While referring the reader to the results of \cite{AlMu2001}, we point out that nontrivial modifications are needed for our setup and we show how these modifications are obtained.

Of particular interest is the connection of the functionals \eqref{eqn:energy} with the Ohta--Kawasaki theory of diblock copolymers (cf. \cite{C2001,ChoksiOctober2003,OK}) on the surface of the unit two-sphere \cite{ChToTs15,To2013}. Pattern formation of ordered structures on curved surfaces arises in systems ranging from biology to materials science. These include covering virus and radiolaria architecture, colloid encapsulation for possible drug delivery. As for the study of diblock copolymers, the self-assembly in thin melt films confined to the surface of a sphere was investigated computationally in \cite{GC,Li_et_al2006} via a model that uses the self-consistent mean field theory. In \cite{Tang} the authors look at the patterns emerging as a result of phase separation of diblock copolymers numerically on spherical surfaces by using the Ohta--Kawasaki model.

Ohta--Kawasaki theory asserts that minimization of the energy
	   \beqn\label{eqn:nlip}
			\OK_\e(u) = \frac{1}{2} \int_\Stwo \e |\nabla u| + \int_\Stwo  |\nabla v|^2\,d\Htwo,
		\eeqn
over $BV(\Stwo;\{\pm1\})$ subject to the mass constraint $\int_\Stwo u\,d \Htwo = 4\pi m$, describes the pattern formation of diblock copolymers. Here $m\in(-1,1)$ is a constant, $\Stwo$ denotes the two-sphere in $\R^3$, $\Htwo$ denotes the two-dimensional Hausdorff measure, and $\int |\nabla u|$ is the total variation of the function $u$. The two phases of these copolymers are described by the phase parameter $u$ taking on values $-1$ and $1$. The function $v$ in the energy \eqref{eqn:nlip} is a solution of the Poisson problem
		\[
			-\Delta v = u-m,
		\]
where $-\Delta$ denotes the Laplace--Beltrami operator on $\Stwo$.  There is extensive literature on the mathematical analysis of the Ohta--Kawasaki model on flat domains, such as the flat-tori, general bounded domains, and the unbounded Euclidean space (see \cite{ChMuTo2017} for a review). 
 The rigorous mathematical analysis of the Ohta--Kawasaki model on curved spaces is rather rare \cite{ChToTs15,To2013}.

When the mass fraction $m$ is zero (i.e., equal amounts of the phases $1$ and $-1$), numerical computations reveal almost striped (or spiral-like) patterns (cf. \cite{GC,Li_et_al2006,Tang}). Approximate striped patterns of diblock copolymers confined in a ball that exhibit different scales depending on the height in the sample have also been observed in experiments \cite{Hi2017}. In order to analyze such patterns we can make an axisymmetric ansatz on the critical pattern (i.e., that a critical pattern $u$ is a function of the polar angle $\phi$ on $\Stwo$ only). Then the energy \eqref{eqn:nlip} of such an axisymmetric pattern becomes
  \beqn\label{eqn:1d_poly}
		\OK_\e(u)= \int_{-1}^1 \left( \e\sqrt{1-t^2}|u^{\pr\pr}| + \frac{u^2}{1-t^2} \right)\,dt
	\eeqn
by a change of variables $t=\cos\phi$ (see \cite[Section 2]{ChToTs15} for details). After another suitable change of variables, minimization of $\E_\e$ with $\alpha=1/2$ and $\beta=1$ is equivalent to minimization of energies $\tld{\E}_\e$ over $H^1$-functions which satisfy $|u^\pr|=1$ and $u(-1)=u(1)=0$.

We conclude by noting that throughout we will use lower and upper case letters $c$ and $C$ (possibly with subscripts as in $C_1$, $C_2$) to denote generic constants which might change from line to line. When necessary we will denote the dependence of a constant to a particular parameter by using the standard function notation such as in $C(\gamma)$.

\bigskip

\section{Scaling of the energies $\E_\e$ and their diffuse-level counter parts}\label{sec:scaling}

In addition to the enegies $\E_\e$ we also consider their diffuse-level counterparts given by the  functionals
\beqn\label{eqn:energy_diff}
	 \F_\e(u) = \int_{0}^1\left( \e^2 \, t^\alpha\,  (u^{\pr\pr})^2 +W(u^\pr)+ t^{-\beta}\,u^2\right)\,dt
\eeqn
over $H^2([0,1])$ where $W$ denotes the double-well potential $W(x):=\frac{1}{4}(1-x^2)^2$. Heuristically it is clear that when $\e$ is sufficiently small the first two terms in the energy $\e^{-1}\F_\e$ approximate $\int_0^1 |u^{\pr\pr}|\,dt$ for any function with $|u^\pr|=1$.

Our main results in this section are the scaling laws for the minimal energies of the functionals $\E_\e$ and $\F_\e$. Namely, we have the following theorems.

\begin{theorem}\label{thm:1}
  Let $\alpha\in\R$ and $\beta<3$. Let $\tilde{u}$ be a minimizer of $\E_\e$ in the class $\A$. Then there are constants $0< C_1 \leq C_2$ such that
  \[
		C_1\,\e^{2/3} \leq \E_\e(\tilde{u}) \leq C_2\, \e^{2/3}.
	\]
\end{theorem}

\medskip

\begin{theorem}\label{thm:2}
  Let $\alpha\in\R$ and $\beta<3$. Let $\tilde{u}\in H^2([0,1])$ be a minimizer of $\F_\e$ subject to $u(0) = u(1) =0$. Then there are constants $0< C_1 \leq C_2$ such that
  \[
		C_1\,\e^{2/3} \leq \F_\e(\tilde{u}) \leq C_2\, \e^{2/3}.
	\]
\end{theorem}

\bigskip

We start with the proof of the first theorem. The main difficulty here is to construct a good upper bound which would compensate the contribution of the singularity $t^{-\beta}$ to the energy.

\begin{proof}[Proof of Theorem \ref{thm:1}] Let us start with the construction of the upper bound.
Fix $n \geq 2$ and let
	\beqn \label{constr1}
		\gamma>-1
	\eeqn
be a constant to be chosen later. We are going to construct a piecewise linear test function $\tilde{u}$ where $\tilde{u}^{\pr}$ has $n$-many jumps located at the points $\{z_k \colon 1 \leq k \leq n\}$. In addition, we will set $z_0 = 0$ and $z_{n+1} = 1$. Now let $z_k$ be such that
\begin{equation}\label{spacing}
y_k := z_k-z_{k-1} = (\gamma+1) \frac{k^\gamma}{n^{\gamma+1}}.
\end{equation}
In particular,	$z_k = \sum_{i=1}^{k} y_i$. Using a generalization of the classic Faulhaber's formula \cite{McGoPa07} given by
\beqn\label{Fau}
\sum_{k=1}^{n} k^{\mu} = \frac{n^{\mu+1}}{\mu+1} + O(n^{\mu}) \qquad \text{for any } \mu>-1,
\eeqn
we get that there is a constant $C(\mu)>0$ such that
\begin{equation}\label{sum}
\sum_{k=1}^{n} k^{\mu} \leq C(\mu) \, n^{\mu+1}
\end{equation}
for any $n \geq 2$. Also, as a consequence of \eqref{Fau}, we have $
\sum_{k=1}^{n} y_k = 1 + O\left(1/n\right)$.
If needed, we can modify the last term $y_{n+1}$ in order to enforce $z_{n+1}=1$, thus
$\sum_{k=1}^{n+1} y_k = 1$ Now we define $\tilde{u}$ on $[0,1]$ as the piecewise linear function of slopes $\pm 1$ where the jumps in the derivative of $\tilde{u}$ occur at $z_k$'s (see Figure~\ref{fig:u_tilde} for an example).

\begin{figure}[ht!]
	\begin{center}
		\includegraphics[width=0.7\linewidth]{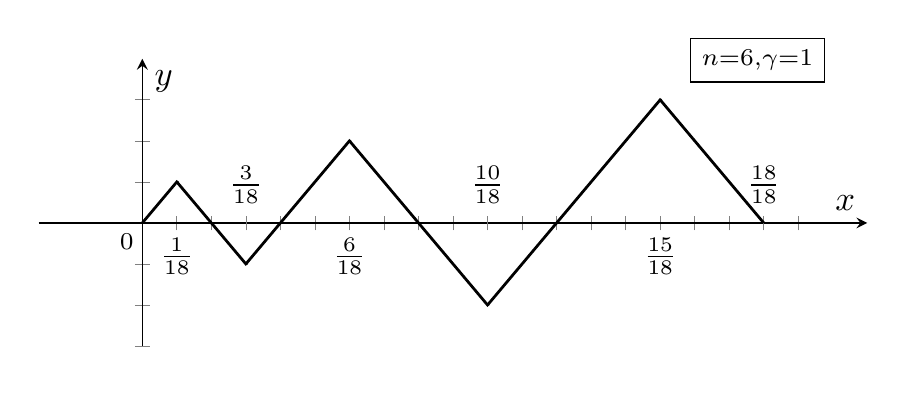}
		\caption{\footnotesize{The function $\tilde{u}$ with six jumps in its derivative, and $\gamma=1$.}} \label{fig:u_tilde}
	\end{center}
\end{figure}

For $t \in [0,z_1]$, we have
\[
\int_{0}^{z_1} t^{-\beta}\, \tilde{u}^2\, dt = \int_{0}^{z_1} t^{2-\beta}\, dt = \frac{z_1^{3-\beta}}{3-\beta}.
\]
provided
\begin{equation}\label{constr2}
\beta<3.
\end{equation}
For $k \geq 2$ the test function $\tilde{u}$ satisfies	$\max_{t \in [z_{k-1},z_k]}|\tilde{u}| = \max\{|\tilde{u}(z_{k-1})|,|\tilde{u}(z_k)|\} \leq y_k$. Therefore, for $k \geq 2$, we have
\[
\int_{z_{k-1}}^{z_{k}} t^{-\beta}\, \tilde{u}^2 \, dt \leq (z_{k-1})^{-\beta} y_k^3   = \left(\sum_{i=1}^{k-1} y_i \right)^{-\beta} y_k^3.
\]
Using \eqref{sum}, for all $k \geq 2$,
\[
\sum_{i=1}^{k-1} y_i = (\gamma + 1) \sum_{i=1}^{k-1} \frac{i^{\gamma}}{n^{\gamma + 1}} \leq  C(\gamma) \frac{k^{\gamma+1}}{n^{\gamma+1}}.
\]
Applying \eqref{sum} once again with
\begin{equation}\label{constr3}
\mu=  3\gamma - \beta(\gamma+1)>-1,
\end{equation}
there exists some constant $C_1(\gamma)>0$ such that
	\begin{align*}
		\int_{0}^{1} t^{-\beta}\, \tilde{u}^2 \, dt &=  \sum_{k=1}^{n+1} \int_{z_{i-1}}^{z_i} t^{-\beta} \tilde{u}^2 dt \leq \frac{z_1^{3-\beta}}{3-\beta} + C(\gamma)^{-\beta} \sum_{k=2}^{n+1}\frac{k^{-\beta(\gamma+1)}}{n^{-\beta(\gamma+1)}} \frac{k^{3\gamma}}{n^{3(\gamma+1)}} \\
													  &\leq \frac{\gamma+1}{3-\beta} \frac{1}{n^{(\gamma+1)(3 - \beta)}} + C_1(\gamma) \frac{n^{-\beta(\gamma+1)+ 3\gamma + 1}}{n^{-\beta(\gamma+1)+ 3(\gamma + 1)}}\\
													  &=\frac{\gamma+1}{3-\beta} \frac{1}{n^{(\gamma+1)(3 - \beta)}} + \frac{C_1(\gamma)}{n^2} = \frac{C_1(\gamma)}{n^2} + O\left(\frac{1}{n^2}\right)
	\end{align*}
provided
\begin{equation}\label{constr4}
(\gamma+1)(3 - \beta) \geq 2.
\end{equation}
In addition, for
\begin{equation}\label{constr5}
\alpha > -\frac{1}{1+\gamma},
\end{equation}
we have
	\begin{align*}
		\int_{0}^{1} \e \, t^\alpha \,  |\tilde{u}^{\pr\pr}| \, dt &= 2 \e\, \sum_{k=1}^{n} (z_k)^{\alpha} = 2 \e \, \sum_{k=1}^n \left(\sum_{i=1}^k y_i\right)^{\alpha} \\
																				   &\leq \e\, C(\gamma) \sum_{k=1}^{n} \left(\frac{k^{\gamma+1}}{n^{\gamma+1}}\right)^{\alpha} \leq \e\, C^2(\gamma) \frac{n^{\alpha(\gamma + 1)+1}}{n^{\alpha(\gamma+1)}} \leq \e\, C^2(\gamma)\, n.
	\end{align*}	
Hence, for $\alpha, \beta$ and $\gamma$ satisfying \eqref{constr1}, \eqref{constr2}, \eqref{constr3}, \eqref{constr4} and \eqref{constr5} (see Figure~\ref{fig:gamma_beta_domain}), we have
	\[
		\E_\e(\tilde{u}) \leq \e \,C^2(\gamma)\, n + \frac{C_1(\gamma)}{n^2}.
	\]
Optimizing in $n$, we see that $n = C \e^{-1/3}$. Therefore $\E_\e(\tilde{u}) \leq C_2 \e^{2/3}$ for some constant $C_2>0$.

\begin{figure}[ht!]
     \begin{center}
     	\includegraphics[width=0.45\linewidth]{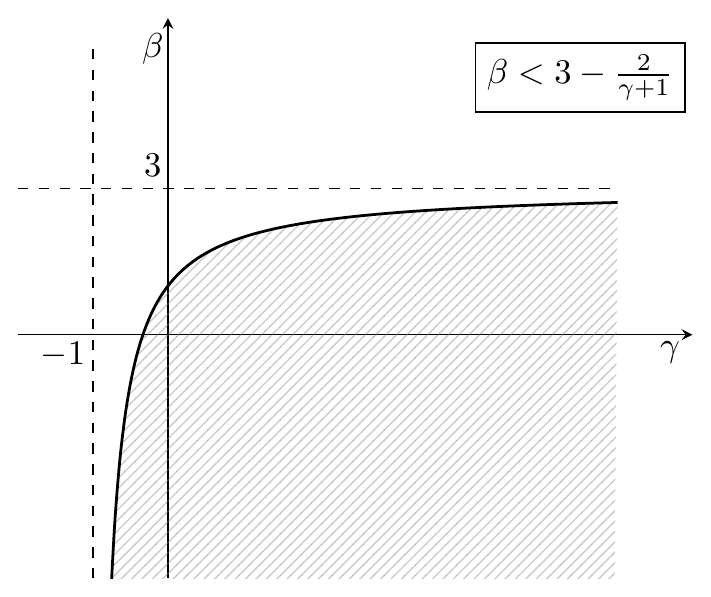}
     \end{center}
    \caption{\footnotesize{
        The shaded area shows the possible choices for $\gamma$ and $\beta$ satisfying the constraints \eqref{constr1}, \eqref{constr2}, \eqref{constr3}, and \eqref{constr4}.}
     }
   \label{fig:gamma_beta_domain}
\end{figure} 

The lower bound follows by repeating the calculations in \cite{Mul93,Yip06} on the interval $[\delta,1]$ for any $0<\delta<1$. Namely, we have that
	\[
		\E_\e(\tilde{u}) \geq  C(\delta) \int_{\delta}^1\left( \e \,   |\tilde{u}^{\pr\pr}| + \tilde{u}^2\right)\,dt \geq C_1 \e^{2/3}
	\]
for some constants $C(d),\, C_1>0$ independent of $\e$. This concludes the proof of the theorem.
\end{proof}

\bigskip

We now turn to the proof of the second theorem.

\begin{proof}[Proof of Theorem \ref{thm:2}] As before, we start with the upper bound. Following the proof of Theorem \ref{thm:1}, for fixed $n \geq 2$ we are going to construct a test function $\hat{u}$ as a modification of the function $\tilde{u}$ obtained in the proof above. For $\mu>0$ define
\beqn 
f_{\mu}(t) =
\begin{dcases*}
\frac{t^2}{2 \mu} & \quad if  $|t| \leq \mu$, \\
|t| - \frac{\mu}{2} & \quad if $|t| \geq \mu$
\end{dcases*}
\eeqn
so that $f_{\mu}(x) \in C^1(-1,1)$ (see Figure~\ref{fig:f_mu_and_u_hat}). We may now define $\hat{u}$ as
\beqn \label{smoothTF}
\hat{u}(t) =
\begin{dcases*}
\pm f_{\mu}(t) + c_k & \quad for $|t-z_k|<\mu$,  $1 \leq k \leq n$,\\
\tilde{u}(t) & \quad otherwise,
\end{dcases*}
\eeqn
where the sign $\pm$ and the constants $c_k$ are chosen in order for $\hat{u} \in C^1([0,1])$, and $\tilde{u}$ is the test function from the Theorem \ref{thm:1}.

\begin{figure}[ht!]
     \begin{center}
     	\includegraphics[height=3.8cm]{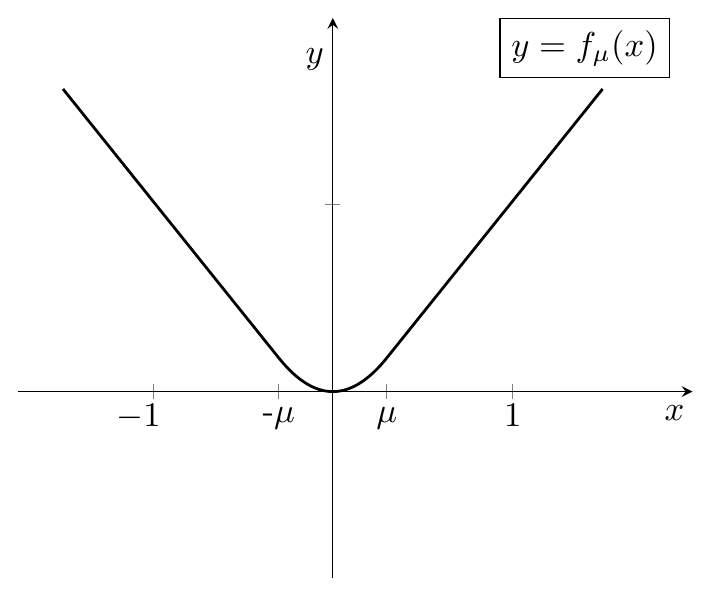} \, \includegraphics[height=3.8cm]{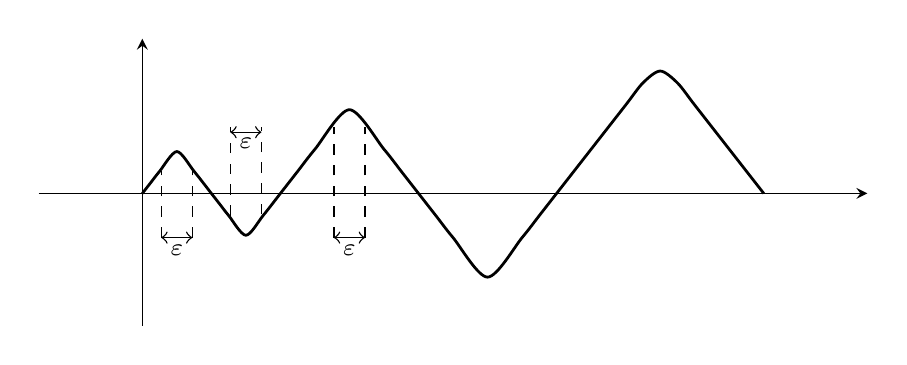}
     \end{center}
    \caption{\footnotesize{
        The functions $f_\mu$ and $\hat{u}$ with the optimal transition layer scale given by $\mu=\e$.}
     }
   \label{fig:f_mu_and_u_hat}
\end{figure} 

 In addition to the conditions for $\alpha, \beta$ and $\gamma$ in Theorem \ref{thm:1}, the natural restriction for the construction \eqref{smoothTF} is that $z_1 \gg \mu$.

By construction we have that
\begin{equation}\label{est1}
\int_{0}^{1} t^{-\beta}\,\hat{u}^2 dt \leq \int_{0}^{1} t^{-\beta}\,\tilde{u}^2 dt \leq \frac{C}{n^2}.
\end{equation}
Therefore it remains to estimate the first two terms in the energy,
\[
\int_{0}^1\left( \e^2 \, t^\alpha\,  (\hat{u}^{\pr\pr})^2 + W(\hat{u}^\pr)\right)\,dt.
\]
Note that
\[
\int_{0}^1 W(\hat{u}^\pr)\, dt = \int_{0}^{1} (1 - (\hat{u}^\pr)^2)^2 \,dt = 2 n \, \int_{z_1}^{z_1+\mu} \left(1 - \left(\frac{t}{\mu}\right)^2 \right)^2 dt  = C \, n\, \mu.
\]
Next we estimate
	\begin{align*}
		\int_{0}^{1} \e^2\, t^{\alpha} (\hat{u}^{\pr\pr})^2 \, dt &= \frac{\e^2}{\mu^2} \sum_{k=1}^{n} \int_{z_k-\mu}^{z_k+\mu} t^{\alpha}\, dt = \frac{\e^{2}}{(\alpha+1) \mu^2} \sum_{k=1}^{n}\Big[(z_k+\mu)^{\alpha+1} - (z_k-\mu)^{\alpha+1}\Big] \\
																			 &\leq \frac{C(\alpha)  \, \e^2}{\mu} \sum_{k=1}^{n} z_k^{\alpha} \leq  \frac{C(\alpha)\, n\,\e^2}{\mu}
	\end{align*}
for some $C(\alpha)>0$ depending on $\alpha$. Hence
\[
\int_{0}^1\left( \e^2 \, t^\alpha\,  (\hat{u}^{\pr\pr})^2 + W(\hat{u}^\pr)\right)\,dt \leq  \frac{C(\alpha)\, n\,\e^2}{\mu} +  C \, n\, \mu.
\]
First we optimize in $\mu$ and get that $\mu = \e$ (see, again, Figure~\ref{fig:f_mu_and_u_hat}). This yields the estimate
\[
\int_{0}^1\left( \e^2 \, t^\alpha\,  (\hat{u}^{\pr\pr})^2 + W(\hat{u}^\pr)\right)\,dt \leq C(\alpha) \,\e\, n.
\]
Combining this with the estimate \eqref{est1} we get
\[
\F_\e(\hat{u}) \leq  \frac{C}{n^2} + C(\alpha) \, \e \, n.
\]
Finally optimizing in $n$ yields the upper bound
\[
\F_\e(\hat{u}) \leq C_2 \e^{2/3}
\]
for some $C_2>0$ independent of $\e$.

\medskip

The lower bound follows as in the proof of Theorem \ref{thm:1} above by chopping off a small interval $[0,\delta]$ containing the singularity, and estimating the energy $\F_\e$ from below using the computations in the literature \cite{Mul93,Yip06}.
\end{proof}

\bigskip

\section{Uniform energy distribution of minimizers of $\E_\e$}\label{sec:distribution}

In this section we prove that the energy of a minimizer is distributed uniformly across the domain although our upper bound construction in the previous section show more frequent oscillations between the phases $\pm 1$ closer to the boundary of the domain at zero.

Let $u_{*}$ be the global minimizer of $\E_\e$. For $x \in [0,1]$ let $\varphi(x)$ be the energy of the minimizer  of $\E_\e$ given on the interval $t \in [0,x]$. Namely,
	\[
		\varphi(x) := \E_\e^x(u_*) \qquad \text{where} \qquad \E_\e^x(u_*) := \int_{0}^x \left( \e \, t^{\alpha} |{u_*}^{\pr\pr}| + t^{-\beta} \, {u_*}^2 \right)\,dt.
	\]
Then our main result in this section states that $\varphi$ grows linearly in $x$.

\begin{theorem}\label{thm:3}
Let $\beta<3$. There exist absolute constants $c_1>0$ and $c_2>0$, such that if
\begin{equation}\label{cond on x}
x \geq c_1 \, \e^{2/(9 -3 \beta)}
\end{equation}
and
\begin{equation}\label{relation}
2\alpha \geq \beta,
\end{equation}
then
\[
\varphi(x) \leq c_2 \,x\, \e^{2/3}.
\]
\end{theorem}

\medskip

For the proof of Theorem \ref{thm:3} we need the following Lemmas. The first lemma below is also a crucial tool in the next section.

\begin{lemma}\label{lem:1}
  For any $a>0$ let
 \beqn\label{eqn:energy:a}
	  e_\e^a := \inf_{\A^a}\E^a_\e(v)  \qquad \text{ with } \quad \E^a_\e(v):=  \int_{0}^a \left( \e \, t^\alpha\,  |v^{\pr\pr}| + t^{-\beta}\,v^2\right)\,dt
	\eeqn
over the admissible class
	\[
		\A^a := \big\{v\in H^2([0,a])  \colon |v^\pr|=1,\ v(0)=v(a)=0\big\}.
	\]
Then
\[
e_\e^a \leq C\, a^{(3+2\alpha - \beta)/3}\, \e^{2/3} 
\]
for some constant $C>0$.
\end{lemma}

\begin{proof}
Rescaling $t = x/a$ and $u(\cdot) = v(\cdot)/a$, for any $v \in \A^{a}$ we have $\pm 1 = v_{x} = v_t/a = u_t$.
Hence, $u(t) \in \A = \A^{1}$ and
\[
\G^a_\e(u):=\int_{0}^1 \left( \e \, a^{\alpha} t^\alpha\,  |u^{\pr\pr}| + a^{3-\beta} t^{-\beta}\,u^2\right)\,dt \\
					= \int_{0}^a \left( \e \, x^\alpha\,  |v^{\pr\pr}| + x^{-\beta}\,v^2\right)\,dx = \E^a_\e(v).
\]
We may now use the test function $\tilde{u} \in \A$ from Theorem \ref{thm:1} with $n$-many ``teeth'' to conclude
\[
\inf_{\A} \G^a_\e[u] \leq \G^a_\e[\tilde{u}] \leq C_1 \e a^{\alpha} n + \frac{C_2 a^{3-\beta}}{n^2}
\]
where $C_1$ and $C_2$ are independent of $\e$ and $n$. Optimizing in $n$, we get $n = C \e^{-1/3}\, a^{(3 - \beta - \alpha)/3}$. Therefore,
\[
e_\e^a= \G^a_\e[\tilde{u}] \leq  C\, a^{(3+2\alpha - \beta)/3} \, \e^{2/3},
\]
and the result follows.
\end{proof}

\medskip

\begin{lemma} \label{lem:2}
  Let $u_{*}$ be a minimizer of $\E_\e$. Then there is an absolute constant $C>0$ such that for all $x \geq 0$
  \[
  |u_{*}(x)| \leq C \, \e^{2/9} \, x^{\beta/3}.
  \]
\end{lemma}
\begin{proof}
Fix $x>0$ such that $u_{*}(x)>0$ (if $u_*(x) < 0$, the argument is  analogous). Furthermore, the constraints $u_*^\pr \leq 1$ and $u_*(0) = 0$ imply that $u_{*}(x) \leq x$. It follows from the upper bound in Theorem \ref{thm:1} that
\begin{equation}\label{est2}
\int_{x-u_*(x)}^{x} \frac{u_*^2(t)}{t^{\beta}}\, dt \leq C \e^{2/3}.
\end{equation}
On the other hand, using the constraint $u_{*}^\pr \leq 1$ once again,
\begin{equation}\label{est3}
\int_{x-u_*(x)}^{x} \frac{u_*^2(t)}{t^{\beta}} \,dt \geq \int_{x-u_*(x)}^{x} \frac{(t-x+u_*(x))^2}{t^{\beta}}\, dt \geq  \frac{u_*^3(x)}{3 x^{\beta}}.
\end{equation}
Combining \eqref{est2} and \eqref{est3}, the statement follows.
\end{proof}

\bigskip

We now return to the proof of Theorem \ref{thm:3}.

\begin{proof}[Proof of Theorem \ref{thm:3}]
Note that $\varphi(x)=\min_{\calK_x} \E_\e^x(u)$ where
	\[
		\calK_x := \big\{ u\in H^2([0,x]) \colon |u^{\pr}|=1, \ u(0)=0, \ u(x)=u_*(x) \big\}.
	\]
If $u_{*}(x) = 0$, then
\[
\varphi(x) \leq C\, \e^{2/3} \, x
\]
using Lemma \ref{lem:1} and \eqref{relation}. Now, assume $u_{*}(x) > 0$ (the case $u_{*}(x) < 0$ is treated analogously). In this case we may argue, using minimality of $\varphi$, that
\begin{equation}\label{minimality}
\varphi(x) \leq \E_\e^x(v)
\end{equation}
for some test function $v \in \calK_x$. In particular, choose $v(t)$ as
\beqn \nonumber
v(t) =
\begin{dcases*}
u_{x-u_{*}(x)}(t)			  & \quad if  $t \in [0,x-u_{*}(x))$,\\
t - \big(x - u_*(x)\big) & \quad if $t \in [x-u_{*}(x),x]$,
\end{dcases*}
\eeqn
 where $u_{x-u_{*}(x)}(t)$ is the minimizer of $\E_{\e}^{x - u_*(x)}$ with $\E_{\e}^a$ defined as in \eqref{eqn:energy:a}. 
 
\begin{figure}[ht!]
     \begin{center}
     	\includegraphics[width=0.7\linewidth]{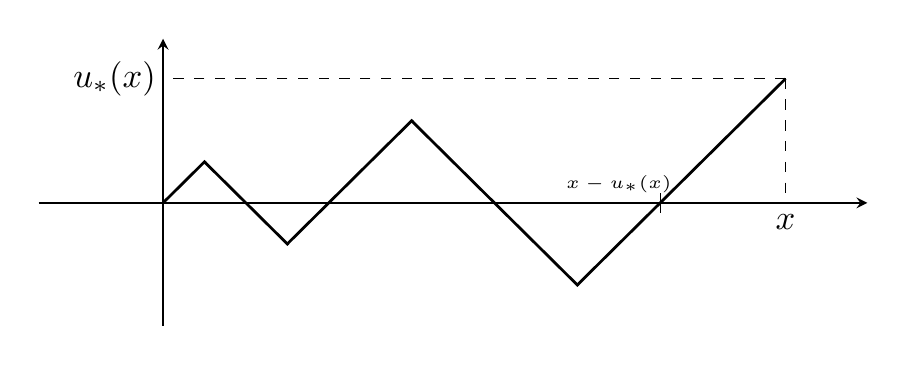}
     \end{center}
    \caption{\footnotesize{
        The function $v$ defined on the interval $[0,x]$ for any $x\in[0,1]$.}
     }
   \label{fig:v}
\end{figure}  

By Lemma \ref{lem:1}, we have that
	\beqn \label{loc:est}
		\begin{aligned}
			\E_\e^x(v) &\leq C \, \e^{2/3}\,(x-u_{*}(x)) + \int_{x-u_{*}(x)}^{x}  \frac{(t - (x - u_*(x)))^2}{t^{\beta}} \,dt \\
								  &\leq C \, \e^{2/3} \, x +  \frac{u_*^3(x)}{3 (x - u_*(x))^{\beta}}.
	\end{aligned}
		\eeqn
On the other hand, it follows from the definition of $\varphi(x)$ that for a.e. $x \in [0,1]$ the function $\varphi(x)$  is differentiable, and for a.e. $x$,
\begin{equation}\label{derivative}
\varphi^\pr(x) = x^{-\beta} u_{*}^2(x).
\end{equation}
Combining \eqref{minimality}, \eqref{loc:est} and \eqref{derivative}, we get
\begin{equation}\label{almost_DE}
\varphi(x) \leq C \, \e^{2/3}\,  x + \frac{u_*(x)}{3} \frac{u_*^2(x)}{(x - u_*(x))^{\beta}} \leq C \e^{2/3} x + \frac{x}{3} \frac{ \varphi^\pr(x) x^{\beta}}{(x - u_*(x))^{\beta}},
\end{equation}
where we used the natural inequality $u_*(x) \leq x$.  By Lemma \ref{lem:2}, $u_*(x) \leq C\, \e^{2/9}$ for all $x \leq 1$. Therefore, if
\[
x \geq \left(\frac{3^{1/3}}{3^{1/3}-1}\right) C\, \e^{2/3}
\]
then we get
\begin{equation}\label{inequal}
x \geq \left(\frac{3^{1/3}}{3^{1/3}-1}\right)  u_{*}(x).
\end{equation}
The latter inequality implies that
\begin{equation}\label{2}
\frac{x}{x-u_{*}(x)} \leq 3^{1/3}.
\end{equation}

In view of \eqref{2}, the inequality \eqref{almost_DE} leads to the following differential inequality:
\[
\varphi(x) \leq C \, \e^{2/3}\, x + \frac{3^{\beta/3}}{3} \, x\, \varphi^\pr(x),
\]
which, in turn, yields for some $c_1>0$,
\begin{equation}\label{improv1}
\varphi(x) \leq c_1\, x\, \e^{2/3} \qquad \text{ for } \quad x \geq \left(\frac{3^{1/3}}{3^{1/3}-1}\right) C \, \e^{2/9},
\end{equation}
since $\beta<3$ by assumption.
For $x \leq  \left(\frac{3^{1/3}}{3^{1/3}-1}\right) C \, \e^{2/9}$, Lemma \ref{lem:2} implies that
\begin{equation}\label{iter1}
u_{*}(x) \leq C_1\, \e^{\frac{2}{9}\left(1+\frac{\beta}{3}\right)}.
\end{equation}
Therefore, for $x \in \left( \left(\frac{3^{1/3}}{3^{1/3}-1}\right) C_1 \e^{\frac{2}{9}\left(1 + \frac{\beta}{3}\right)}, \left(\frac{3^{1/3}}{3^{1/3}-1}\right) C \e^{\frac{2}{9}}\right)$, we obtain \eqref{inequal}. As a result, the estimate \eqref{improv1} extends to this larger interval, i.e., for some $c_2>0$ we have
\begin{equation}\label{improv2}
\varphi(x) \leq c_2 \, x\, \e^{2/3} \qquad \text{ for } \quad x \geq \left(\frac{3^{1/3}}{3^{1/3}-1}\right) \, C_1 \, \e^{2/9(1 + \beta/3)}.
\end{equation}
The inequality \eqref{improv2}, in turn, implies, by Lemma \ref{lem:2},
\begin{equation}\label{iter3}
u_{*}(x) \leq C_2 \, \e^{\frac{2}{9}\left(1+\frac{\beta}{3} + \left(\frac{\beta}{3}\right)^2\right)}
\end{equation}
for $x \leq \left(\frac{3^{1/3}}{3^{1/3}-1}\right) \, C_1 \, \e^{2/9(1 + \beta/3)}$. Iterating this process, for all $n \geq 1$ we have
\[
u_{*}(x) \leq C \, \e^{\frac{2}{9}\left(\sum_{k=1}^{n + 1}\left(\frac{\beta}{3}\right)^k\right)}
\]
as long as $x \leq C \, \e^{\frac{2}{9}\left(\sum_{k=1}^{n}\left(\frac{\beta}{3}\right)^k\right)}$, and thus
\begin{equation}\label{improv3}
\varphi(x) \leq c_2\, x\,  \e^{2/3} \qquad \text{ for } \quad  x \geq C \e^{\frac{2}{9}\left(\sum_{k=1}^{n}\left(\frac{\beta}{3}\right)^k\right)}.
\end{equation}
Passing to the limit as $n \to \infty$ in \eqref{improv3}, the statement of the theorem follows.
\end{proof}

\bigskip

\section{Asymptotic description of minimizers of $\F_\e$}\label{sec:limit}

In this section we prove our next main result describing the minimizing patterns of diffuse-level energies $\F_\e$ in the $\e\to 0$ limit.

\bthm\label{thm:4}
Let $\alpha, \, \beta>0$ such that $\beta-2\alpha<3$. Let $\{u_\e\}_{\e>0}$ be a sequence of minimizers of diffuse-level energies $\F_\e$. Let $R_s^\e(u)(t):= \e^{-1/3} u_\e(s+\e^{1/3}t)$ be the rescalings of $u_\e$ and let $\nu$ be the Young measure which arises as the limit of the maps $s\mapsto R_s^\e u_\e$ as $\e\to 0$. Then for a.e. $s\in[0,1]$ the measure $\nu_s$ is supported on the set of all translations of sawtooth functions 
	\beqn \label{eqn:sawtooth}
		y_{h(s)}(t) = |t|- h/4 \qquad \text{for} \quad t\in \left( -h/2,h/2\right]
	\eeqn
with slope $\pm 1$ and period $h(s):=L\, s^{(\alpha+2\beta)/6}$ where $L:=\big(96 \int_{-1}^1 \sqrt{W} \big)^{1/3}$.
\ethm

We establish this theorem by closely following the arguments in Alberti and M\"{u}ller's work \cite[Chapter 3]{AlMu2001} which relies on a $\Gamma$-convergence argument. In their paper Alberti and M\"{u}ller consider two-scale energies where the weight in front of the elastic term is in $L^\infty$. However, they note that, with some modification, their results would apply to cases where the weight is in $L^1$. Obtaining these modifications for our functional $\E_\e$ is the main goal in this section.

As stated in \cite[Chapter 3]{AlMu2001}, the proof of Theorem \ref{thm:4} requires several steps. The first step is to identify the class of all Young measures $\nu$ that are generated by sequences of $\e$-blowups of functions $u_\e$. We refer the reader to \cite[Chapter 2]{AlMu2001} for details regarding Young measures and the space thereof. For the convenience of the reader we state the next result as a lemma; the proof appears in the paper of Alberti and M\"{u}ller.

\blemma[Proposition 3.1 in \cite{AlMu2001}]\label{lem:Young_meas}
	Let $\nu$ be a Young measure generated by the $\e$-blowups of a countable sequence of measurable functions $u_\e$. Then for a.e. $s\in(0,1)$, $\nu_s$ is an invariant measure on the space of all measures functions $x:\R \to [-\infty,\infty]$ modulo equivalence almost everywhere.
\elemma

The next step involves rewriting the functional $\F_\e(u)$ in terms of the rescaled functions $R_s^\e u$. For simplicity of notation let us denote
	\[
		a(t):= t^\alpha \qquad \text{and} \qquad b(t):= t^{-\beta},
	\]
and by extending the functions $a$ and $b$ periodically out of $(0,1)$ we set
	\[
		a_s^\e(t) := (s+\e^{1/3}t)^\alpha \qquad \text{and} \qquad b_s^\e(t) := (s+\e^{1/3}t)^{-\beta}.
	\]
For any function $u\in H^2_{\per}([0,1])$, let $x_s = R_s^\e u$ for every $s\in [0,1]$. Also, for any fixed $r>0$, and for any function $x$ of class $H^2$ on $(-r,r)$ we set
	\beqn\label{eqn:f_s_eps}
		f_s^\e(x) := \fint_{-r}^r \Big( \e^{2/3}a_s^\e(t)(x^{\pr\pr})^2 + \eps^{-2/3} W(x^{\pr}) + b_s^\e(t) x^2 \Big)\,dt.
	\eeqn
Note that, with this definition we have
	\[
		\e^{-2/3} \F_\e (u) = \int_0^1 f_s^\e (x_s) \,ds.
	\]
	
The third step in Alberti and M\"{u}ller's program describes the asymptotic behavior of the functionals $f_s^\e$ as $\e\to 0$. We state this result in a lemma below.

\blemma \label{lem:f_gamma_conv}
	Let $\alpha,\, \beta>0$. Let $s\in(0,1)$. Then the sequence of functionals $f_s^\e$, extended to be $+\infty$ for any measurable function $x$ on $(-r,r)$ that is not in $H^2(-r,r)$, $\Gamma$-converges to the functional
		\beqn \label{eqn:f_s}
			f_s(x) := \begin{dcases*}
							A_0 \, \sqrt{a(s)} \fint_{-r}^r |x^{\pr\pr}|\,dt+b(s) \fint_{-r}^r x^2\,dt & \quad if $x\in \Se(-r,r)$, \\
							+\infty & \quad otherwise.
						 \end{dcases*}
		\eeqn
where $A_0=2\int_{-1}^1 \sqrt{W}$ and $\Se(-r,r)$ denotes the class of sawtooth functions, i.e., all measurable functions, modulo equivalence almost everywhere, that are continuous and piecewise affine on $(-r,r)$ with slope $\pm 1$.
\elemma

The proof of this lemma relies on a general $\Gamma$-convergence result for anisotropic Cahn--Hilliard-type functionals by Owen and Sternberg \cite[Section 3]{OwSt91}. We state a special case of their result here for the convenience of the readers.

\bprop[Example 1 in Section 3 in \cite{OwSt91}] \label{prop:Owen_Sternberg} 
The sequence of functionals 
	\beqn
		\begin{dcases*}
							\int_0^1 \Big(\e A(x,u)|\nabla u|^2 + \frac{1}{\e}W(u)\Big)\,dx & \quad if $u\in H^1_{\text{per}}([0,1])$, \\
							+\infty & \quad otherwise,
						 \end{dcases*}	
	\eeqn
$\Gamma$-converges with respect to the $L^1$-topology to the functional
	\beqn
		\begin{dcases*}
							\int_0^1 \int_{-1}^1 \sqrt{W(s)A(x,s)}\,dx d|u|(x) & \quad  if $u\in BV_{\text{per}}([0,1])$, and $|u|=1$ a.e.,\\
							+\infty & \quad otherwise,
						 \end{dcases*}	
	\eeqn
as $\e \to 0$, where $d|u|$ denotes the total variation measure.	
\eprop

\bigskip

Now we turn to the proof of the lemma above.

\begin{proof}[Proof of Lemma \ref{lem:f_gamma_conv}]
Note that we cannot directly apply the above proposition to the first two terms of $f_s^\e$ since the weight in the first term of $f_s^\e$ also depends on $\e$. However, since $a_s^{\e}(\cdot)=a(s+\e^{1/3}\,\cdot)$ by definition, and since $a(t)=|t|^\alpha$, we have that $a_s^\e=a(s)+O(\e^\alpha)$ and $a_s^\e \to a(s)$ in $L^1(0,1)$ for any fixed $s\in(0,1)$ as $\e\to 0$. Therefore, after a diagonal argument, the functionals
	\[
		\fint_{-r}^r \Big( \e^{2/3}a_s^\e(t)(x^{\pr\pr})^2 + \eps^{-2/3} W(x^{\pr}) \Big)\,dt
	\]
$\Gamma$-converge, in the topology of $W^{1,1}(-r,r)$, to the functional $A_0 \sqrt{a(s)} \fint_{-r}^r |x^{\pr\pr}|\,dt$ for functions $x\in\Se(-r,r)$ by the above proposition. As argued in the proof of \cite[Proposition 3.3]{AlMu2001}, the additional elastic term is a continuous perturbation due to the fact that $b_s^\e \to b(s)$ in $L^1_{\loc}(0,1)$ for any $s\in(0,1)$. Since $W^{1,1}(-1,1)$ embeds continuously into the set of measurable functions modulo translations, the result follows from \cite[Proposition 2.11(v)]{AlMu2001}.
\end{proof}

\bigskip

Before we can state the main $\Gamma$-convergence result, we need to introduce some notation. Let $\M$ denote the set of all measurable functions $x:\R \to [-\infty,\infty]$ modulo equivalence almost everywhere, and let $\Y((0,1),\M)$ denote the set of all $\M$-valued Young measures on $(0,1)$. Again, we refer the readers to \cite[Chapter 2]{AlMu2001} for a concise introduction to the theory of Young measures.

Define
	\beqn	\label{eqn:H_eps}
		\He_\e(\nu) := \begin{dcases*}
							\int_0^1 \ip{\nu_s}{f_s^\e}\, ds & \quad if $\nu$ is the elementary Young measure \\ 
																	& \quad  associated to $R^\e v$ for some $v\in H^2_{\text{per}}([0,1])$, \\
																	\\
							+\infty & \quad otherwise,
						 \end{dcases*}	
	\eeqn
where $\ip{\mu}{f}:=\int_{\M} f \,d\nu$ denotes the duality pairing.
Note that $\He_\e(\nu)$ is finite if and only if $\nu$ is the elementary Young measure associated with the $\e$-blowup $R^\e v$ of some $v\in H_{\text{per}}^2([0,1])$ and $\He_\e(v) = \e^{-2/3} \F_\e(v)$.

Let $\Inv(\M)$ denote the space of all probability measures on $\M$ that are invariant under translations, and define
	\beqn	\label{eqn:H}
		\He (\nu) := \begin{dcases*}
							\int_0^1 \ip{\nu_s}{f_s}\, ds & \quad if $\nu_s\in \Inv(\M)$ for a.e. $s\in(0,1)$. \\
							+\infty & \quad otherwise.
						 \end{dcases*}	
	\eeqn 
With these definitions we have the following proposition.
	\bprop \label{prop:main_gamma_conv}
		Let $\alpha,\, \beta>0$ such that $\beta-2\alpha <3$. The sequence of functionals $\He_\e$ $\Gamma$-converges to $\He$ on $\Y((0,1),\M)$.
	\eprop
	
As noted in \cite[Remark 3.5]{AlMu2001}, Alberti and M\"{u}ller establish this convergence result for a large class of problems; however, the upper bound estimates rely on the definition of their version of $f_s^\e$. In the proof below, we will show the modifications needed for the upper bound of the energies.

\begin{proof} As per definition of $\Gamma$-convergence, we need to establish a lower bound and an upper bound inequality.

\medskip

\noindent \emph{Lower bound.} The lower bound inequality, namely, that $\liminf_{\e\to 0} \He_\e(\nu_\e) \geq \He(\nu)$ for any sequence $\nu_\e \to \nu$ in $\Y((0,1),\M)$ can be obtained by the arguments in \cite{AlMu2001} verbatim. Assuming that the left-hand side of the inequality is finite, and that the limit is attained (possibly after passing to a subsequence), we see that by the definition of the energy $\He_\e$ each $\nu_\e$ has to be the elementary Young measure associated to some $\e$-blowup. Since $\nu_\e\to \nu$, by Lemma \ref{lem:Young_meas}, $\nu_s$ is an invariant measure for a.e. $s\in(0,1)$. Thus the lower bound inequality becomes $\liminf_{\e\to 0} \int_0^1 \ip{\nu_s^\e}{f_s^\e}\,ds \geq \int_0^1 \ip{\nu_s}{f_s}\,ds$. This inequality follows from Lemma \ref{lem:f_gamma_conv} and \cite[Theorem 2.12(iv)]{AlMu2001}.

\medskip

\noindent \emph{Upper bound.} Let $\D$ denote the class of Young measures $\nu\in\Y((0,1),\M)$ such that there exists countably many disjoint intervals that cover almost all of $(0,1)$, and on every such interval $\nu$ agrees a.e. with an elementary invariant measure $\epsilon_x$ with $x\in\Se_{\per}(0,h)$ and $h>0$, depending on the interval. Then \cite[Lemma 3.8]{AlMu2001} proves that $\D$ is $\He$-dense in $\Y((0,1),\M)$. That is, for every $\nu\in\Y((0,1),\M)$ such that $\He(\nu)$ is finite, there exists $\nu^k\in\D$ such that $\nu^k \to \nu$ in $\Y((0,1),\M)$ and $\limsup_{k\to\infty} \He(\nu^k) \leq \He(\nu)$. By \cite[Remark 2.10]{AlMu2001}, in order to establish the upper bound inequality, it suffices to show that every measure $\nu\in\D$ can be approximated in energy by $\e$-blowups of some functions on $(0,1)$. This can be established in two steps: First by approximating a constant Young measure $\nu$, and next obtaining a localization of such an approximation to a general Young measure in $\D$.

Now let $I$ be a given bounded interval, $x\in\Se_{\per}(0,h)$ for some $h>0$, and $x^\e \in H^2_{\per}(0,h)$ be a sequence of functions converging to $x$ in $\M$ and satisfying
	\beqn \label{eqn:limsup_fs_trans}
		\limsup_{\e\to 0} \fint_{\substack{\tau\in[0,h]\\s\in I}} f_s^\e(T_\tau x^\e)\,d\tau ds \leq \fint_{s\in I} \ip{\epsilon_x}{f_s}\,ds + \eta
	\eeqn
for a given $\eta>0$ where $T_\tau x(t) = x(t-\tau)$ denotes the translation operator.
For every $\e>0$ choose $\tau^\e \in [0,h]$, and set
	\beqn \label{eqn:v_eps_trans}
		v^\e(s) := \e^{1/3} x^\e (\e^{-1/3}s - \tau^\e) \qquad \text{ for every } s\in\R.
	\eeqn
Then, by \cite[Lemma 3.9]{AlMu2001}, $v^\e\in H^2_{\per}(0,h\e^{1/3})$, and the $\e$-blowups $R^\e v^\e$ generate on $I$ the constant Young measure $\epsilon_x$. Also, the numbers $\tau^\e$ in \eqref{eqn:v_eps_trans} can be chosen so that
	\beqn	\label{eqn:limsup_fs_avrg}
		\limsup_{\e \to 0} \fint_I f_s^\e (R^\e v^\e)\,ds \leq \fint_{I} \ip{\epsilon_x}{f_s}\,ds + \eta.
	\eeqn
A crucial fact in the proof of the upper bound is that every $\nu\in\D$ can be approximated locally by a sequence of $\e$-blowups. Namely, by \cite[Lemma 3.10]{AlMu2001}, for any $\nu\in \D$ and any $\eta>0$, there exists finitely many intervals $I_i$ with pairwise disjoint closures that cover $(0,1)$ with $|(0,1)\setminus \bigcup I_i|<\eta$ such that there exists $h_i>0$ and $x_i \in \Se_{\per}(0,h_i)$ satisfying $\nu_s = \eps_{x_i}$ for a.e. $s\in I_i$, and for every $\e>0$ there exists a 1-Lipschitz function $x_i^\e \in H^2_{\per}(0,h_i)$ that converges to $x_i$ in $\M$ and satisfies
	\[
		\limsup_{\e\to 0} \fint_{\substack{\tau\in[0,h_i]\\s\in I_i}} f_s^\e(T_\tau x_i^\e)\,d\tau ds \leq \fint_{s\in I_i} \ip{\epsilon_{x_i}}{f_s}\,ds + \eta.
	\]

\medskip

Since $\D$ is $\He$-dense in $\Y((0,1),\M)$, as Alberti and M\"{u}ller note, by \cite[Remark 2.10]{AlMu2001} it suffices to construct, for every $\delta>0$ and $\nu\in\D$, functions $w^\e\in H^2_{\per}(0,1)$ so that the elementary Young measures $\nu^\e$ associated with the $\e$-blowups $R^\e w^\e$ satisfy
	\beqn	\label{eqn:upper_bd_main}
		\limsup_{\e\to 0} \int_0^1 f_s(R_s^\e w^\e)\,ds \leq \int_0^1 \ip{\nu_s}{f_s}\,ds + \delta.
	\eeqn

\medskip

Now, fix $\nu\in\D$ and $\delta>0$. Let $\eta>0$ be a constant that will be chosen later.  Define $v_i^\e$ as in \eqref{eqn:v_eps_trans} via the functions $x_i^\e$ where $\tau_i^\e$ are chosen so that \eqref{eqn:limsup_fs_avrg} is satisfied. Let us denote the intervals $I_i$ by $(a_i,b_i)$ where $a_i<b_i<a_{i+1}<b_{i+1}$. We set
	\beqn \label{eqn:v_eps_outside}
		v^\e(s) := v_i^\e(s) \qquad \text{ if } s\in (a_i+r\e^{1/3}, b_i-r\e^{1/3}) \text{ for some } i,
	\eeqn
where $r$ is the constant in the definition of $f_s$. Now we will describe how to extend the function $v^\e$ out of the union of the intervals $(a_i+r\e^{1/3},b_i-r\e^{1/3})$. Let $M>\max\{1,r\}$ such that $|x_i(t)|+1 \leq M$ for every $i$ and every $t\in\R$. Since $x_i^\e \to x_i$ in $\M$ and are 1-Lipschitz they also converge uniformly. Therefore for $\e$ sufficiently small $|x_i^\e(t)|\leq M$ for every $i$ and $t$ , and $|v^\e(s)|\leq M \e^{1/3}$. For $\e$ sufficiently small, extend $v^\e$ to the interval $[b_i-r\e^{1/3},a_{i+1}+r\e^{1/3}]$ so that $(v^\e)^{\pr}$ alternates between the values $\pm 1$ in a sequence of intervals of length of order $\e^{1/3}$ except the first and the last one which have length of order $M\e^{1/3}$. We take $(v^\e)^{\pr\pr}$ of order $\e^{-1}$ on the transition layers with length of order $\e$ which separate two consecutive intervals where $|(v^\e)^{\pr}|=1$. Due to construction the value of $v^\e$ is of order $\e^{1/3}$ in each interval except the first and the last one where it is of order $M\e^{1/3}$. Finally, we note that since the weight $b(s)$ is singular at the origin, we have to modify the construction of Alberti and M\"{u}ller near $s=0$. 

Fix $\theta>0$, to be chosen later, and let
	\beqn	\label{eqn:upper_bd_test_func}
		w^\e(s) := \begin{dcases*}
							\hat{u}(s/\theta) & \quad if $s\in [0,\theta]$, \\
							\pm(s-\theta) & \quad if $s\in[\theta,\theta_M)$, \\
							v^\e(s) & \quad if $s\in [\theta_M,1]$,
						 \end{dcases*}	
	\eeqn 
where $\hat{u}$ is defined by \eqref{smoothTF} in the proof of Theorem \ref{thm:2}, and $\theta_M$ is the solution of
	\[
		v^\e(\theta_M) = \pm (\theta_M-\theta).
	\]
Note that due to the a priori bound on $v^\e$, we have that $\theta \leq \theta_M \leq \theta+\e^{1/3} M$.


Also, as shown by Alberti and M\"{u}ller,  the elementary Young measures $\nu^\e$ generated by the blowups $R^\e w^\e$ approximate $\nu$ in some norm that metrizes the space $\Y((0,1),\M)$. We refer the reader to equation (2.2) in \cite{AlMu2001} for the definition of this norm.

Now, to prove \eqref{eqn:upper_bd_main}, let us write
	\[
		\int_0^1 f_s^\e (R_s^\e w^\e)\,ds = \int_0^{\theta} f_s^\e(R_s^\e w^\e)\,ds + \int_{\theta}^{\theta_M} f_s^\e(R_s^\e w^\e)\,ds + \int_{\theta_M}^1 f_s^\e(R_s^\e w^\e)\,ds.
	\]
Arguing as in the proof of \cite[Theorem 3.4]{AlMu2001}, it is easy to see that the estimates (3.29)--(3.32) of Alberti and M\"{u}ller hold in our case, since we are away from the singularity at $s=0$, with an additional weight of $\theta^{-\beta}$. Namely,
	\beqn \label{eqn:first_est}
		\int_0^1 f_s^\e (R_s^\e w^\e)\,ds \leq  \int_0^{\theta} f_s^\e(R_s^\e w^\e)\,ds + \sum_{i} \int_{I_i} f_s^\e(R_s^\e v_i^\e)\,ds + \frac{1}{\theta^\beta} \Big( C_1 \eta + C_2 M^3 \e^{1/3} \Big)
	\eeqn
for some constants $C_1,\ C_2>0$.

On the other hand, for $s\in[0,\theta]$, we have that either $R_s^\e w^\e(s) = \theta \hat{u}(\theta^{-1}(s+\e^{1/3}t))$ or $R_s^\e w^\e(s) = \pm ((s+\e^{1/3}t)-\theta)$. Thus, for $\e$ sufficiently small, a direct computation after a change of variables shows that for $s\in[0,\theta]$,
	\begin{align*}
		f_s^\e (R_s^\e w^\e) &\leq \frac{1}{2r \e^{2/3}} \int_0^{\theta_M} a(\xi) \big(w^\e(\xi)\big)^{\pr\pr} + W\big((w^\e(\xi))^{\pr}\big) + b(\xi) \big( w^\e(\xi)\big)^2\,d\xi \\
									   &= 	\frac{1}{2r \e^{2/3}} \left( \int_0^\theta \quad + \int_{\theta}^{\theta_M} \quad \right) \\
									   &\leq \frac{1}{2r \e^{2/3}} \Big( \theta^{(3+2\alpha-\beta)/3}\e^{2/3} \Big) + \frac{1}{2r \e^{2/3}}\frac{M^3 \e}{\theta^\beta}.
	\end{align*}
In the last line of the above estimate, the first term is a consequence of Lemma \ref{lem:1} and the second term follows after a direct calculation of the energy contribution. Note that, by assumptions on $\alpha$ and $\beta$, we have $3+2\alpha-\beta>0$.
		
Returning to \eqref{eqn:first_est}, we have that
	\[
		\int_0^1 f_s^\e (R_s^\e w^\e)\,ds \leq \sum_{i} \int_{I_i} f_s^\e(R_s^\e v_i^\e)\,ds + \frac{1}{\theta^\beta} \Big( C_1 \eta + C_2 M^3 \e^{1/3} \Big) + C_3 \theta^{(3+2\alpha-\beta)/3}.
	\]	
We let $\e\to 0$, and obtain, by \eqref{eqn:limsup_fs_avrg},
	\[
		\limsup_{\e\to 0} \int_0^1 f_s^\e(R_s^\e w^\e)\,ds \leq \int_0^1 \ip{\nu_s}{f_s}\,ds + \frac{C_1 \eta}{\theta^\beta}+ C_3 \theta^{(3+2\alpha-\beta)/3}.
	\]
Letting $\theta = \eta^{1/(2\beta)}$ yields
	\[
		\limsup_{\e\to 0} \int_0^1 f_s^\e(R_s^\e w^\e)\,ds \leq \int_0^1 \ip{\nu_s}{f_s}\,ds + C_1 \eta^{1/3}+ C_3 \eta^{(3+2\alpha-\beta)/(6\beta)}.
	\]
Hence, 	by choosing $C_1 \eta^{1/3}+ C_3 \eta^{(3+2\alpha-\beta)/(6\beta)} < \delta$ we obtain \eqref{eqn:upper_bd_main}.
\end{proof}

\medskip

Before we turn to the proof of Theorem \ref{thm:4} we would like to note that a classical consequence of $\Gamma$-convergence of the energies $\He_\e$ is the following result.

\begin{corollary}
	Let $\alpha, \, \beta>0$ such that $\beta-2\alpha<3$. For every $\e>0$, let $u^\e$ be a minimizer of $\F_\e$ on $H^2_{\per}([0,1])$, and let $\nu$ be a Young measure generated by a subsequence of the $\e$-blowups $R^\e u^\e$. Then $\nu$ minimizes the functional $\He$ given by \eqref{eqn:H}. That is, for a.e. $s\in(0,1)$ the measure $\nu_s$ minimizes $\ip{\mu}{f_s}$ among all invariant probability measures $\mu$ on $\M$. 
\end{corollary}

In fact, every Young measure generated by the $\e$-blowups of the minimizers of $\F_\e$ is uniquely determined by the minimality property stated in the corollary above. For $h>0$, if we define $y_h$ to be the $h$-periodic function on $\R$ given by \eqref{eqn:sawtooth}	
then the following statement holds.

\bprop\label{prop:sawtooth}
	Let $s\in(0,1)$. If $\bar{\mu}$ minimizes $\ip{\mu}{f_s}$ among all invariant probability measures on $\M$, then $\bar{\mu}$ is the elementary invariant measure associated with the function $y_{h(s)}$ where
		\[
			h(s) := L\, s^{(\alpha+2\beta)/6} \quad \text{ with } L:=\left(96 \int_{-1}^1 \sqrt{W} \right)^{1/3}.
		\]
\eprop

As a consequence of this proposition, combined with the corollary above, we have the following result which completes Theorem \ref{thm:4}.

\begin{corollary}
	Let $\alpha, \, \beta>0$ such that $\beta-2\alpha<3$. For every $\e>0$, let $u^\e$ be a minimizer of $\F_\e$ on $H^2_{\per}([0,1])$. Then the $\e$-blowups $R^\e u^\e$ generate a unique Young measure $\nu \in \Y((0,1),\M)$, and a.e. $s\in(0,1)$, $\nu_s$ is the elementary invariant measure associated with the sawtooth function $y_{h(s)}$.
\end{corollary}

Proposition \ref{prop:sawtooth} above is only a modified version of \cite[Theorem 3.12]{AlMu2001} by Alberti and M\"{u}ller. The proof of their theorem applies to our case with a minor modification as the majority of the arguments are independent of the specific form of the functionals $f_s$. The only modification is needed in the computation of $\ip{\eps_x}{f_s}$. Namely, with the inclusion of a weight in front of the singular perturbation term, in Alberti and M\"{u}ller's notation, we have that
	\[
		\ip{\eps_x}{f_s} = \sum_{i=1}^n \frac{h_i}{h} g(h_i,p_i),
	\]
where $g(h,p)= \frac{A_0 \sqrt{a(s)}}{h} + \frac{b(s) h^2}{12} + b(s) p^2$. Since $g$ is convex in $p$, the minimum occurs when $p=0$. Then optimizing in $h$ we see that the unique minimum occurs when $h(s) = (48 A_0 \sqrt{a(s)})^{1/3} (b(s))^{-1/3}= (48A_0)^{1/3} s^{(\alpha+2\beta)/6}$.


\bibliographystyle{IEEEtranS}
\def\url#1{}
\bibliography{lit}

\begin{thebibliography}{10}
\providecommand{\url}[1]{#1}
\csname url@samestyle\endcsname
\providecommand{\newblock}{\relax}
\providecommand{\bibinfo}[2]{#2}
\providecommand{\BIBentrySTDinterwordspacing}{\spaceskip=0pt\relax}
\providecommand{\BIBentryALTinterwordstretchfactor}{4}
\providecommand{\BIBentryALTinterwordspacing}{\spaceskip=\fontdimen2\font plus
\BIBentryALTinterwordstretchfactor\fontdimen3\font minus
  \fontdimen4\font\relax}
\providecommand{\BIBforeignlanguage}[2]{{%
\expandafter\ifx\csname l@#1\endcsname\relax
\typeout{** WARNING: IEEEtranS.bst: No hyphenation pattern has been}%
\typeout{** loaded for the language `#1'. Using the pattern for}%
\typeout{** the default language instead.}%
\else
\language=\csname l@#1\endcsname
\fi
#2}}
\providecommand{\BIBdecl}{\relax}
\BIBdecl

\bibitem{AlMu2001}
\BIBentryALTinterwordspacing
G.~Alberti and S.~M\"uller, ``A new approach to variational problems with
  multiple scales,'' \emph{Comm. Pure Appl. Math.}, vol.~54, no.~7, pp.
  761--825, 2001.  \url{https://doi.org/10.1002/cpa.1013}
\BIBentrySTDinterwordspacing

\bibitem{GC}
T.~L. Chantawansri, A.~W. Bosse, A.~Hexemer, H.~D. Ceniceros, C.~J.
  Garc{\'\i}a-Cervera, E.~J. Kramer, and G.~H. Fredrickson, ``Self-consistent
  field theory simulations of block copolymer assembly on a sphere,''
  \emph{Physical Review E}, vol.~75, no.~3, p. 031802, 2007.

\bibitem{C2001}
\BIBentryALTinterwordspacing
R.~Choksi, ``Scaling laws in microphase separation of diblock copolymers,''
  \emph{J. Nonlinear Sci.}, vol.~11, no.~3, pp. 223--236, 2001.
  \url{http://dx.doi.org/10.1007/s00332-001-0456-y}
\BIBentrySTDinterwordspacing

\bibitem{ChToTs15}
\BIBentryALTinterwordspacing
R.~Choksi, I.~Topaloglu, and G.~Tsogtgerel, ``Axisymmetric critical points of a
  nonlocal isoperimetric problem on the two-sphere,'' \emph{ESAIM: COCV},
  vol.~21, no.~1, pp. 247--270, 2015.
  \url{http://dx.doi.org/10.1051/cocv/2014031}
\BIBentrySTDinterwordspacing

\bibitem{ChMuTo2017}
R.~Choksi, C.~B. Muratov, and I.~Topaloglu, ``An old problem resurfaces
  nonlocally: {G}amow's liquid drops inspire today's research and
  applications,'' \emph{Notices Amer. Math. Soc.}, vol.~64, no.~11, pp.
  1275--1283, 2017.

\bibitem{ChoksiOctober2003}
R.~Choksi and X.~Ren, ``On the derivation of a density functional theory for
  microphase separation of diblock copolymers,'' \emph{Journal of Statistical
  Physics}, vol. 113, no. 1/2, pp. 151--176, 2003.

\bibitem{Co2000}
\BIBentryALTinterwordspacing
S.~Conti, ``Branched microstructures: scaling and asymptotic self-similarity,''
  \emph{Comm. Pure Appl. Math.}, vol.~53, no.~11, pp. 1448--1474, 2000.
  \url{https://doi.org/10.1002/1097-0312(200011)53:11<1448::AID-CPA6>3.0.CO;2-C}
\BIBentrySTDinterwordspacing

\bibitem{Hi2017}
T.~Higuchi, ``Microphase-separated structures under spherical 3d confinement,''
  \emph{Polymer Journal}, vol.~49, no.~6, p. 467, 2017.

\bibitem{KoMu94}
\BIBentryALTinterwordspacing
R.~V. Kohn and S.~M\"{u}ller, ``Surface energy and microstructure in coherent
  phase transitions,'' \emph{Comm. Pure Appl. Math.}, vol.~47, no.~4, pp.
  405--435, 1994.  \url{https://doi.org/10.1002/cpa.3160470402}
\BIBentrySTDinterwordspacing

\bibitem{Li_et_al2006}
\BIBentryALTinterwordspacing
{Li, J. F.}, {Fan, J.}, {Zhang, H. D.}, {Qiu, F.}, {Tang, P.}, and {Yang, Y.
  L.}, ``Self-assembled pattern formation of block copolymers on the surface of
  the sphere using self-consistent field theory,'' \emph{Eur. Phys. J. E},
  vol.~20, no.~4, pp. 449--457, 2006.
  \url{https://doi.org/10.1140/epje/i2006-10035-y}
\BIBentrySTDinterwordspacing

\bibitem{McGoPa07}
\BIBentryALTinterwordspacing
K.~J. McGown and H.~R. Parks, ``The generalization of {F}aulhaber's formula to
  sums of non-integral powers,'' \emph{J. Math. Anal. Appl.}, vol. 330, no.~1,
  pp. 571--575, 2007.  \url{https://doi.org/10.1016/j.jmaa.2006.08.019}
\BIBentrySTDinterwordspacing

\bibitem{Mul93}
\BIBentryALTinterwordspacing
S.~M\"uller, ``Singular perturbations as a selection criterion for periodic
  minimizing sequences,'' \emph{Calc. Var. Partial Differential Equations},
  vol.~1, no.~2, pp. 169--204, 1993.  \url{https://doi.org/10.1007/BF01191616}
\BIBentrySTDinterwordspacing

\bibitem{OK}
T.~Ohta and K.~Kawasaki, ``Equilibrium morphology of block copolymer melts,''
  \emph{Macromolecules}, vol.~19, no.~10, pp. 2621--2632, 1986.

\bibitem{OwSt91}
\BIBentryALTinterwordspacing
N.~C. Owen and P.~Sternberg, ``Nonconvex variational problems with anisotropic
  perturbations,'' \emph{Nonlinear Anal.}, vol.~16, no. 7-8, pp. 705--719,
  1991.  \url{https://doi.org/10.1016/0362-546X(91)90177-3}
\BIBentrySTDinterwordspacing

\bibitem{Tang}
P.~Tang, F.~Qiu, H.~Zhang, and Y.~Yang, ``Phase separation patterns for diblock
  copolymers on spherical surfaces: A finite volume method,'' \emph{Physical
  Review E}, vol.~72, no.~1, p. 016710, 2005.

\bibitem{To2013}
\BIBentryALTinterwordspacing
I.~Topaloglu, ``On a nonlocal isoperimetric problem on the two-sphere,''
  \emph{Commun. Pure Appl. Anal.}, vol.~12, no.~1, pp. 597--620, 2013.
  \url{http://dx.doi.org/10.3934/cpaa.2013.12.597}
\BIBentrySTDinterwordspacing

\bibitem{Yip06}
\BIBentryALTinterwordspacing
N.~K. Yip, ``Structure of stable solutions of a one-dimensional variational
  problem,'' \emph{ESAIM Control Optim. Calc. Var.}, vol.~12, no.~4, pp.
  721--751, 2006.  \url{https://doi.org/10.1051/cocv:2006019}
\BIBentrySTDinterwordspacing

\end{thebibliography}

\end{document}